\documentclass[12pt,table]{article}

\usepackage{amssymb}
\usepackage{pdfpages}
\usepackage{graphicx}
\usepackage{floatrow}
\usepackage{amsmath}
\usepackage{booktabs}
\usepackage{amsfonts}
\usepackage[pagebackref=true,colorlinks,linkcolor=red,citecolor=blue]{hyperref}
\hypersetup{
  colorlinks=true,
  urlcolor=blue}

\newtheorem{theorem}{Theorem}

\newtheorem{definition}[theorem]{Definition}

\newtheorem{question}{Question}
\newtheorem{example}{Example}
\newtheorem{lemma}[theorem]{Lemma}

\newtheorem{proposition}[theorem]{Proposition}
\newenvironment{proof}{\noindent {\bf Proof.}}{\hfill\rule{3mm}{3mm}\par\medskip}

\begin{document}
\title{Perfect Pseudo-Matchings in cubic graphs\footnote{
Research supported  by FWF Project P27615-N25.
}}
\author{{\sc Herbert Fleischner${}^{a}$}, {\sc Behrooz Bagheri Gh.${}^{a,b}$},\\ and {\sc Benedikt  Klocker${}^{a}$}}
\date{}
  \maketitle
  \vspace{-1cm}
\begin{center}
$a$
{\small \it Algorithms and Complexity Group}\\
{\small \it Vienna University of Technology}\\
{\small  \it Favoritenstrasse 9-11,}
\vspace*{5mm}
{\small \it 1040 Vienna, Austria } \\
$b$
{\small \it Department of Mathematics} \\
{\small \it West Virginia University} \\
{\small \it WV 26506-6310,}
\vspace*{5mm}
{\small \it Morgantown, USA} \
\end{center}



\begin{abstract}
A {\sf perfect pseudo-matching} $M$ in a cubic graph $G$ is a spanning subgraph of $G$ such that every component of $M$ is isomorphic to $K_2$ or to $K_{1,3}$. In view of snarks $G$ with dominating cycle $C$, this is a natural generalization of perfect matchings since $G \setminus E(C)$ is a perfect pseudo-matching. Of special interest are such $M$ where $G/M$ is planar because such $G$ have a cycle double cover. We show that various well known classes of snarks contain planarizing perfect pseudo-matchings, and that there are at least as many snarks with planarizing perfect pseudo-matchings as there are cyclically $5-$edge-connected snarks.

\vspace{1cm}

\end{abstract}

\section{Introduction and preliminaries}

All concepts not defined in this paper can be found in~\cite{Bondy2008,FleischnerBook1, Fleischner_badK5}, giving preference to a definition as stated in~\cite{FleischnerBook1} if it differs from the corresponding definition in~\cite{Bondy2008}.    
A cycle $C$ in a graph $G$ is {\sf dominating} if $E(G\setminus V(C))=\emptyset$. 
A cycle $C$ in a graph $G$ is called {\sf stable} if there exists no other cycle $D$ in $G$ such that $V(C)\subseteq V(D)$.

Our point of departure is the following.

\medskip \noindent
{\bf Sabidussi's Compatibility Conjecture (SC Conjecture)} {\it Given a connected eulerian graph G with $\delta(G) > 2$ and an eulerian trail $T_\epsilon$ of $G$, there is a cycle decomposition $\mathcal{S}$ of $G$ such that edges consecutive in $T_\epsilon$ belong to different elements in $\mathcal{S}$. Whence one calls $\mathcal{S}$ {\sf compatible} with $T_\epsilon$.}\\

The converse of the SC Conjecture is easily proved (see~\cite{Kotzig} but also~\cite[Theorem~VI.1]{FleischnerBook1}): Given a cycle decomposition  $\mathcal{S}$ in the above $G$, then $G$ has an eulerian trail $T_\epsilon$ such that $\mathcal{S}$ and $T_\epsilon$ are compatible. As has been noted before it suffices to consider such eulerian graphs $G$ for which $4 \le \delta(G) \le \Delta(G) \le 6$,~\cite[Lemma~1]{Fleischner1984}. In turn, such $G$ can be transformed into a cubic graph by splitting away from each vertex the transitions of a given eulerian trail $T_\epsilon$ of $G$, thus transforming $T_\epsilon$ into a cycle $C_\epsilon$. This transformation of $T_\epsilon$ into $C_\epsilon$ can also be viewed as a detachment of $G$ (see, e.g.,~\cite[Corollaries~V.10~and~V.13]{FleischnerBook1}). Now, for a  vertex $v$ of degree $4$ in $G$ and the corresponding vertices $v^{'}$ and $v^{''}$ in $C_\epsilon$ add an edge $v^{'}v^{''}$, whereas for a  vertex $w$ of degree $6$ in $G$ and the corresponding $w^{'}, w^{''}, w^{'''}$ in $C_\epsilon$ introduce a new vertex $w^{*}$  and the edges $w^{*}w^{'}, w^{*}w^{''}, w^{*}w^{'''}$. The resulting cubic graph $G_3$ contains $C_\epsilon$ as a dominating cycle. Call $(G_3, C_\epsilon)$ {\sf associated} with $(G, T_\epsilon)$. It is intuitively clear how one obtains $(G, T_\epsilon)$ from $(G_3, C_\epsilon)$. The following is well known~\cite[pp.~236--237]{Fleischner1984}.

\begin{proposition}\label{proposition1}
Let $(G, T_\epsilon)$ and $(G_3, C_\epsilon)$ be as above. $G$ has a cycle decomposition compatible with $T_\epsilon$ if and only if $G_3$ has a cycle double cover $\mathcal{S}$ with $C_\epsilon \in \mathcal{S}$.
\end{proposition}

 \medskip \noindent
{\bf Cycle Double Cover Conjecture (CDC Conjecture)} {\it In every bridgeless graph $G$ with $E(G)\neq  \emptyset$ there is a collection $\mathcal{S}$ of cycles such that every edge of $G$ belongs to exactly two elements of $\mathcal{S}$.}
\\

In dealing with the CDC Conjecture, it suffices to consider snarks, i.e., cyclically $4-$edge-connected cubic graphs which are not $3-$edge-colorable.
In our understanding of snarks we omit the usual requirement that the girth must exceed $4$.
 For snarks, the following conjecture has been formulated which is equivalent to various other conjectures, \cite{Zdenek2008,Fleischner_Bill,Zdenek1997}.\\

 \medskip \noindent
{\bf Dominating Cycle Conjecture (DC Conjecture)} {\it Every snark has a dominating cycle.}\\

 Proposition~\ref{proposition1} and the DC Conjecture demonstrate the close relationship between SC Conjecture and CDC Conjecture.
 
\begin{definition}
 Let $G$ be  a $2-$connected
 eulerian graph.
 For  each  vertex $v\in V(G)$,  let
$\mathcal{T}(v)$  be a set
 of disjoint edge-pairs of $E(v)$, and,
  $\mathcal{T}
 = \bigcup_{v \in V(G)} \mathcal{T}(v)$. Then,  $\mathcal{T}$ is called a {\sf transition system}, and a
  cycle decomposition $\mathcal{C}$ of $G$ is
{\sf compatible} with $\mathcal{T}$
 if $|E(C)\cap P|\le 1$ for every member
 $C\in \mathcal{C}$ and every
 $P \in \mathcal{T}$.
\end{definition} 
 
 For planar eulerian graphs $G$ one can prove a more general result than the SC Conjecture; namely it suffices to assume that the transition system is non-separating; one even has a compatible cycle decomposition (CCD) in  an arbitrary eulerian graph $G$  if the given transition system in $G$
 is non-separating and does not yield a {\rm SUD}-$K_5-$minor{\footnote{See the definition of a {\rm SUD}-$K_5-$minor in a transitioned graph $(G,\mathcal{T})$ in~\cite{Fleischner_badK5}.}}. This has been shown recently in~\cite{Fleischner_badK5}.

 We note explicitly that the compatible cycle decomposition problem has been verified
for planar graphs
  by
Fleischner~\cite{Fleischner1980},
 for $K_5-$minor-free graphs by
Fan and Zhang~\cite{Fan2000ccd}, 
and for {\rm SUD}-$K_5-$minor-free graphs by
Fleischner  {\it et al.}~\cite{Fleischner_badK5}.

On the other hand, when focusing on the SC Conjecture with all vertices of degree $4$ and $6$ only in a given graph $G$ with eulerian trail $T_\epsilon$, one has a direct connection to cubic graphs with a dominating cycle  $C_\epsilon$ where $C_\epsilon$ corresponds to $T_\epsilon$ and viceversa, as noted above (see Proposition~\ref{proposition1}).
This leads us to the following definition. 

\begin{definition}
A {\sf perfect pseudo-matching} $M$ in a graph $G$ is a collection of vertex-disjoint subgraphs of $G$ such that
\begin{itemize}
\item[$\bullet$] every member of $M$ is isomorphic to either $K_2$ or $K_{1,3}$, and
\item[$\bullet$] $\bigcup_{H\in M}V(H)=V(G)$.
\end{itemize}
\end{definition}
\begin{definition} Let $M$ be a perfect matching $($perfect pseudo-matching$)$ in a graph $G$.
 $M$  is called  
 \begin{itemize}
 \item[$\bullet$]{\sf planarizing} if $G/M$ is a planar graph;
 \item[$\bullet$]{\sf $K_5-$minor-free} if $G/M$ has no $K_5-$minor;
 \item[$\bullet$]{\sf $\rm SUD$-$K_5-$minor-free} if $(G/M,\mathcal{T}_M)$ has no $\rm SUD$-$K_5-$minor, where
$\mathcal{T}_M$ is the transition system induced by contracting $M$.
 \end{itemize}
\end{definition}

Note that every perfect matching is a perfect pseudo-matching; and 
for every dominating cycle $C$ in a cubic graph $G$, $M:=G\setminus E(C)$  is also a perfect pseudo-matching.

 Note that $G \setminus E(M)$ is a set $\mathcal{C}_0$ of disjoint cycles in $G$ for any perfect pseudo-matching $M$ of a cubic graph $G$. We conclude that if $G/M$ has a compatible cycle decomposition (with the transitions defined by the pairs of adjacent edges in $\mathcal{C}_0$ and thus yielding a transition system $\mathcal{T}_M$), then $G$ has a cycle double cover $\mathcal{C}$ with $\mathcal{C}_0\subset \mathcal{C}$. 
 
 \begin{example} Consider the perfect pseudo-matching $M$, visualized as bold-face edges  in the Petersen graph $P$ in  Figure~$\ref{FIG:Petersen}$.
 $P\setminus E(M)$ yields the dominating cycle \\ $C_0=v_1v_2v_3v_4v_9v_7v_5v_8v_6v_1$. Then
 $$\mathcal{C}=\{C_0,v_0v_1v_2v_7v_5v_0,v_0v_1v_6v_9v_4v_0,v_0v_4v_3v_8v_5v_0,v_2v_3v_8v_6v_9v_7v_2\}$$
 is a cycle double cover 
 of the Petersen graph  containing $C_0$ obtained from a compatible cycle decomposition of $(P/M,\mathcal{T}_M)$.
 This {\rm CDC} $\mathcal{C}$ exists because $P/M$ is planar; i.e., 
 $M$ is a  planarizing perfect pseudo-matching.
  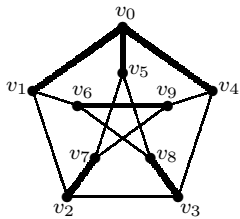
\begin{figure}[ht]

\setlength{\unitlength}{.125cm}
\vspace{1cm}
\begin{center}


\begin{picture}(8,14)
\put(0,0){\circle*{1.2}}
\put(-1.5,-1.75){\scriptsize$v_2$}
\put(11.8,0){\circle*{1.2}}
\put(12.,-1.75){\scriptsize$v_3$}
\put(15.46,11.21){\circle*{1.2}}
\put(16.1,11.2){\scriptsize$v_4$}
\put(5.9,18.1){\circle*{1.2}}
\put(5.1,19.2){\scriptsize$v_0$}
\put(-3.66,11.21){\circle*{1.2}}
\put(-6.5,11.2){\scriptsize$v_1$}

\put(2.95,4.05){\circle*{1.2}}
\put(.2,4){\scriptsize$v_7$}
\put(8.85,4.05){\circle*{1.2}}
\put(9.5,4){\scriptsize$v_8$}

\put(10.7,9.7){\circle*{1.2}}
\put(9.5,10.8){\scriptsize$v_9$}

\put(5.9,13.1){\circle*{1.2}}
\put(6.5,13.){\scriptsize$v_5$}

\put(1.1,9.7){\circle*{1.2}}
\put(.5,10.8){\scriptsize$v_6$}

\thicklines
\linethickness{0.15mm}

\qbezier(0,0)(0,0)(11.8,0)
\qbezier(0,0)(0,0)(-3.66,11.21)
\qbezier(1.1,9.7)(1.1,9.7)(-3.66,11.21)
\qbezier(10.7,9.7)(10.7,9.7)(15.46,11.21)
\qbezier(1.1,9.7)(1.1,9.7)(8.85,4.05)
\qbezier(10.7,9.7)(10.7,9.7)(2.95,4.05)
\qbezier(5.9,13.1)(5.9,13.1)(2.95,4.05)
\qbezier(5.9,13.1)(5.9,13.1)(8.85,4.05)
\qbezier(11.8,0)(11.8,0)(15.46,11.21)


\thicklines
\linethickness{0.6mm}
\qbezier(-3.66,11.21)(5.9,18.)(5.9,18.)
\qbezier(5.9,13.1)(5.9,18.)(5.9,18.)
\qbezier(15.46,11.21)(5.9,18.)(5.9,18.)

\qbezier(1.1,9.7)(1.1,9.7)(10.7,9.7)

\qbezier(0,0)(0,0)(2.95,4.05)
\qbezier(11.8,0)(11.8,0)(8.85,4.05)

\end{picture}


\end{center}
\caption{\small\it 
The  Petersen graph and its planarizing perfect pseudo-matching visualized as bold-face edges.}
\label{FIG:Petersen}
\end{figure}
 \end{example}

\section{Planarizing perfect pseudo-matchings in snarks}

In view of Fleischner's result on compatible cycle decomposition in planar eulerian $G$ with $\delta(G)\ge 4$ and with given non-separating  transition system $\mathcal{T}$,~\cite{Fleischner1980}, it is of particular interest to search for planarizing perfect pseudo-matchings $M$ in snarks. We study various well-known types of graphs in this direction.

By a computer search we found that  all snarks with up to $26$ vertices  contain a planarizing perfect pseudo-matching except two snarks of order $26$;  one of them is given in  Figure~\ref{FIG:Snark26}.

\begin{description}
  
   \item[(a)] {\bf Blanu\v{s}a snarks}

 \begin{figure}[ht]

\setlength{\unitlength}{0.125cm}
\vspace{1cm}
\begin{center}


\begin{picture}(102,14)
\put(5,17){$B_0$}
\put(10,2){\circle*{1.2}}
\put(7.75,0.5){\tiny$u_0$}
\put(10,7){\circle*{1.2}}
\put(7.,7){\tiny$u_3$}
\put(10,12){\circle*{1.2}}
\put(7.75,13){\tiny$u_5$}
\put(15,2){\circle*{1.2}}
\put(14,0.5){\tiny$u_1$}
\put(15,12){\circle*{1.2}}
\put(12.75,13){\tiny$u_6$}
\put(20,2){\circle*{1.2}}
\put(20,0.5){\tiny$u_2$}
\put(20,7){\circle*{1.2}}
\put(20.75,7){\tiny$u_4$}
\put(20,12){\circle*{1.2}}
\put(20,13){\tiny$u_7$}

\put(4,2){\footnotesize$a$}
\put(4,12){\footnotesize$b$}
\put(24.5,2){\footnotesize$b^{'}$}
\put(24.5,12){\footnotesize$a^{'}$}

\qbezier(6,2)(10,2)(24,2)
\qbezier(10,7)(10,7)(20,7)
\qbezier(6,12)(10,12)(24,12)
\qbezier(10,2)(10,7)(10,12)
\qbezier(15,2)(15,7)(15,12)
\qbezier(20,2)(20,7)(20,12)

\put(40,17){$B_1$}
\put(45,2){\circle*{1.2}}
\put(42.75,0.5){\tiny$v_0$}
\put(45,7){\circle*{1.2}}
\put(42.,7){\tiny$v_3$}
\put(45,12){\circle*{1.2}}
\put(42.75,13){\tiny$v_5$}
\put(50,2){\circle*{1.2}}
\put(49,0.5){\tiny$v_1$}
\put(50,12){\circle*{1.2}}
\put(47.75,13){\tiny$v_6$}
\put(55,2){\circle*{1.2}}
\put(54,0.5){\tiny$v_2$}
\put(55,7){\circle*{1.2}}
\put(55.75,7){\tiny$v_4$}
\put(55,12){\circle*{1.2}}
\put(54,13){\tiny$v_7$}
\put(60,2){\circle*{1.2}}
\put(60,.5){\tiny$v_8$}
\put(60,12){\circle*{1.2}}
\put(60,13){\tiny$v_9$}

\put(39,2){\footnotesize$a$}
\put(39,12){\footnotesize$b$}
\put(64.5,2){\footnotesize$b^{'}$}
\put(64.5,12){\footnotesize$a^{'}$}

\qbezier(41,2)(45,2)(64,2)
\qbezier(45,7)(45,7)(55,7)
\qbezier(41,12)(45,12)(64,12)
\qbezier(45,2)(45,7)(45,12)
\qbezier(50,2)(50,7)(50,12)
\qbezier(55,2)(55,7)(55,12)
\qbezier(60,2)(60,7)(60,12)

\put(80,17){$B_2$}
\put(85,2){\circle*{1.2}}
\put(82.75,0.5){\tiny$w_0$}
\put(85,7){\circle*{1.2}}
\put(82.,7){\tiny$w_3$}
\put(85,12){\circle*{1.2}}
\put(82.75,13){\tiny$w_5$}
\put(90,2){\circle*{1.2}}
\put(89,0.5){\tiny$w_1$}
\put(90,12){\circle*{1.2}}
\put(87.75,13){\tiny$w_6$}
\put(95,2){\circle*{1.2}}
\put(95,0.5){\tiny$w_2$}
\put(95,7){\circle*{1.2}}
\put(95.75,7){\tiny$w_4$}
\put(95,12){\circle*{1.2}}
\put(95,13){\tiny$w_7$}
\put(90,9.5){\circle*{1.2}}
\put(87.,9.5){\tiny$w_8$}
\put(92.5,7){\circle*{1.2}}
\put(91.5,5.5){\tiny$w_9$}

\put(79,2){\footnotesize$a$}
\put(79,12){\footnotesize$b$}
\put(99.5,2){\footnotesize$b^{'}$}
\put(99.5,12){\footnotesize$a^{'}$}

\qbezier(81,2)(85,2)(99,2)
\qbezier(85,7)(85,7)(95,7)
\qbezier(81,12)(85,12)(99,12)
\qbezier(85,2)(85,7)(85,12)
\qbezier(90,2)(90,7)(90,12)
\qbezier(95,2)(95,7)(95,12)
\qbezier(90,9.5)(90,9.5)(92.5,7)

\end{picture}


\end{center}
\caption{\small\it 
The Blanu\v{s}a blocks $B_0,B_1$, and $B_2$.}
\label{FIG:BlanusaBlocks}
\end{figure}
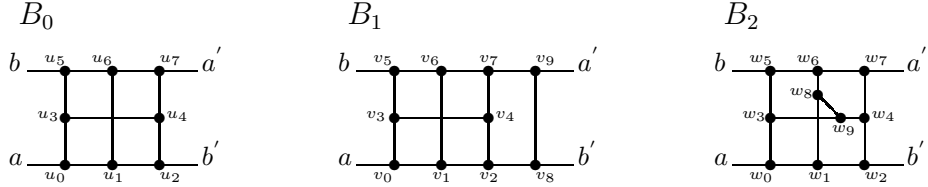

Consider the three subgraphs which we call {\sf Blanu\v{s}a  blocks} $B_0$, $B_1$, and $B_2$ in  Figure~\ref{FIG:BlanusaBlocks}.
For every $n$ we can construct a Blanu\v{s}a snark $B_n^j$, $j=1,2$, as follows,~\cite{Watkins1989, Ghebleh2008}.
Let $H_i$, $1\le i\le n-1$, be $n-1$ copies of the block $B_0$ and let $H_n$ be a copy of  $B_j$, $j=1,2$. Then connect the half-edges  $a^{'}$ and $b^{'}$
of each $H_{i}$ to  the half-edges $a$  and $b$ of the block $H_{i+1}$, respectively, for $i=1,\ldots,n-1$,  and likewise connect $H_n$ to $H_1$.

  Note that $B_1^{1}$ is the Petersen graph, and $B_1^{2}$ is the first Blanu\v{s}a snark.  The set of all members of
  all  $B_n^1$'s and all $B_n^2$'s are called {\sf generalized Blanu\v{s}a snarks}. 
 
 For every $B_n^{1}$ let $M_1$ be a set of all edges of $<u_0,u_1,u_2,u_6>$, $u_3u_5$, and $u_4u_7$ in all copies of $B_0$ and
 the edges of $<v_0,v_1,v_2,v_6>$, $v_3v_5,  v_4v_7$, and $v_8v_9$ in the copy of $B_1$.

 For every $B_n^{2}$ let $M_2$ be a set of all edges of $<u_0,u_3,u_4,u_5>$, $u_1u_2$, and $u_6u_7$ in all copies of $B_0$ and
 the edges $w_0w_1,w_2w_4, w_3w_5,  w_6w_7$, and $w_8w_9$ in the copy of $B_2$.
 
 By  a drawing of  Blanu\v{s}a snarks as exemplified in Figure~\ref{FIG:BlanusaSnarks}  and  by Theorem~\ref{Th:drawing} below, one  can see that $M_j$ is a  planarizing perfect pseudo-matching in $B_n^{j}$, for $j=1,2$.

 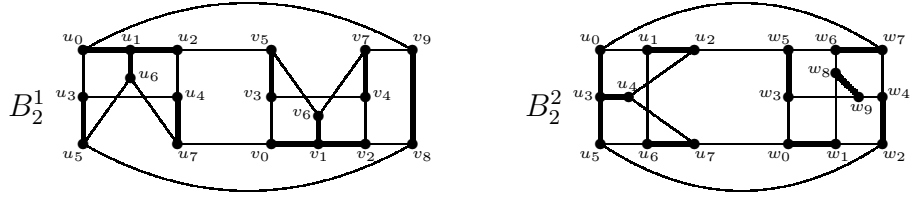
\begin{figure}[ht]

\setlength{\unitlength}{0.125cm}
\vspace{1cm}
\begin{center}


\begin{picture}(102,18)
\put(2,5){$B_2^{1}$}
\put(10,2){\circle*{1.2}}
\put(7.75,0.5){\tiny$u_5$}
\put(10,7){\circle*{1.2}}
\put(7.,7){\tiny$u_3$}
\put(10,12){\circle*{1.2}}
\put(7.75,13){\tiny$u_0$}
\put(15,12){\circle*{1.2}}
\put(14,13){\tiny$u_1$}
\put(15,9){\circle*{1.2}}
\put(15.8,9){\tiny$u_6$}
\put(20,2){\circle*{1.2}}
\put(20,0.5){\tiny$u_7$}
\put(20,7){\circle*{1.2}}
\put(20.75,7){\tiny$u_4$}
\put(20,12){\circle*{1.2}}
\put(20,13){\tiny$u_2$}

\put(30,2){\circle*{1.2}}
\put(27.75,0.5){\tiny$v_0$}
\put(30,7){\circle*{1.2}}
\put(27.,7){\tiny$v_3$}
\put(30,12){\circle*{1.2}}
\put(27.75,13){\tiny$v_5$}
\put(35,2){\circle*{1.2}}
\put(34,0.5){\tiny$v_1$}
\put(35,5){\circle*{1.2}}
\put(32.2,5){\tiny$v_6$}
\put(40,2){\circle*{1.2}}
\put(39,0.5){\tiny$v_2$}
\put(40,7){\circle*{1.2}}
\put(40.75,7){\tiny$v_4$}
\put(40,12){\circle*{1.2}}
\put(38.5,13){\tiny$v_7$}
\put(45,2){\circle*{1.2}}
\put(45,.5){\tiny$v_8$}
\put(45,12){\circle*{1.2}}
\put(45,13){\tiny$v_9$}

\put(57,5){$B_2^{2}$}
\put(65,2){\circle*{1.2}}
\put(62.75,0.5){\tiny$u_5$}
\put(65,7){\circle*{1.2}}
\put(62.,7){\tiny$u_3$}
\put(65,12){\circle*{1.2}}
\put(62.75,13){\tiny$u_0$}
\put(70,12){\circle*{1.2}}
\put(69,13){\tiny$u_1$}
\put(70,2){\circle*{1.2}}
\put(69,.5){\tiny$u_6$}
\put(75,2){\circle*{1.2}}
\put(75,0.5){\tiny$u_7$}
\put(68,7){\circle*{1.2}}
\put(66.7,8){\tiny$u_4$}
\put(75,12){\circle*{1.2}}
\put(75,13){\tiny$u_2$}

\put(85,2){\circle*{1.2}}
\put(82.75,0.5){\tiny$w_0$}
\put(85,7){\circle*{1.2}}
\put(81.8,7){\tiny$w_3$}
\put(85,12){\circle*{1.2}}
\put(82.75,13){\tiny$w_5$}
\put(90,2){\circle*{1.2}}
\put(89,0.5){\tiny$w_1$}
\put(90,12){\circle*{1.2}}
\put(87.75,13){\tiny$w_6$}
\put(95,2){\circle*{1.2}}
\put(95,0.5){\tiny$w_2$}
\put(95,7){\circle*{1.2}}
\put(95.5,7.2){\tiny$w_4$}
\put(95,12){\circle*{1.2}}
\put(95,13){\tiny$w_7$}
\put(90,9.5){\circle*{1.2}}
\put(87.,9.5){\tiny$w_8$}
\put(92.5,7){\circle*{1.2}}
\put(91.5,5.5){\tiny$w_9$}

\thicklines
\linethickness{0.15mm}


\qbezier(30,2)(30,2)(45,2)
\qbezier(10,2)(27.5,-8)(45,2)
\qbezier(10,12)(27.5,22)(45,12)
\qbezier(10,7)(10,7)(20,7)
\qbezier(30,7)(40,7)(40,7)
\qbezier(10,2)(10,2)(15,9)
\qbezier(20,2)(20,2)(15,9)
\qbezier(40,12)(40,12)(45,12)
\qbezier(30,12)(30,12)(35,5)
\qbezier(40,12)(40,12)(35,5)
\qbezier(10,2)(10,7)(10,12)
\qbezier(30,2)(30,7)(30,12)
\qbezier(40,2)(40,7)(40,12)
\qbezier(20,2)(20,7)(20,12)

\qbezier(20,12)(20,12)(30,12)
\qbezier(30,2)(30,2)(20,2)

\qbezier(85,2)(85,2)(95,2)
\qbezier(65,2)(65,2)(65,12)
\qbezier(68,7)(75,2)(75,2)
\qbezier(85,7)(85,7)(95,7)
\qbezier(90,2)(90,7)(90,12)
\qbezier(95,2)(95,7)(95,12)
\qbezier(90,9.5)(90,9.5)(92.5,7)

\qbezier(68,7)(75,12)(75,12)
\qbezier(85,12)(85,12)(95,12)

\qbezier(85,2)(75,2)(75,2)
\qbezier(85,12)(85,12)(75,12)

\qbezier(65,12)(80,22)(95,12)
\qbezier(65,2)(80,-8)(95,2)

\qbezier(70,2)(70,2)(70,12)

\qbezier(65,2)(70,2)(70,2)

\qbezier(70,12)(65,12)(65,12)
\qbezier(85,2)(85,7)(85,12)
\qbezier(90,2)(95,2)(95,2)
\qbezier(95,7)(95,7)(95,12)
\qbezier(85,7)(85,7)(92.5,7)
\qbezier(90,9.5)(90,12)(90,12)

\thicklines
\linethickness{0.6mm}
\qbezier(15,12)(15,9)(15,9)
\qbezier(10,12)(10,12)(20,12)
\qbezier(10,7)(10,7)(10,2)
\qbezier(20,7)(20,7)(20,2)
\qbezier(30,7)(30,7)(30,12)
\qbezier(35,2)(35,5)(35,5)
\qbezier(40,7)(40,7)(40,12)
\qbezier(45,2)(45,7)(45,12)
\qbezier(30,2)(40,2)(40,2)

\qbezier(70,12)(70,12)(75,12)
\qbezier(65,12)(65,12)(65,2)
\qbezier(65,7)(65,7)(68,7)
\qbezier(70,2)(75,2)(75,2)
\qbezier(85,7)(85,7)(85,12)
\qbezier(90,9.5)(90,9.5)(92.5,7)
\qbezier(95,2)(95,7)(95,7)
\qbezier(95,12)(90,12)(90,12)
\qbezier(85,2)(85,2)(90,2)

\end{picture}


\end{center}
\caption{\small\it 
The Blanu\v{s}a snarks $B_2^1$ and $B_2^2$ and their planarizing perfect pseudo-matchings exhibited by the bold-face edges.}
\label{FIG:BlanusaSnarks}
\end{figure}

\item[(b)] {\bf Flower snarks}

For an odd integer $k\ge 3$, the {\sf flower snark} $J_k$ is constructed as follows. 
$V(J_k)=\{v_1,v_2,\ldots,v_k\}\cup \{u_1^1,u_1^2,u_1^3,u_2^1,u_2^2,u_2^3,\ldots,u_k^1,u_k^2,u_k^3\}$.
 $J_k$ is comprised of a cycle $C_1=u_1^1u_2^1\ldots u_k^1u_1^1$ of length $k$ 
and a cycle 
$C_2=u_1^2u_2^2\ldots u_k^2u_1^3u_2^3\ldots u_k^3u_1^2$ of length $2k$, and in addition, each vertex 
$v_i,\ 1\le i\le k$, is adjacent to $u_i^1, u_i^2,$ and $u_i^3$.

Note that for  even  $k\ge 4$, $C_1$ is an even cycle.
Color the edges of $C_1$ with color set $\{2,3\}$; color $u_i^jv_i,\ i=1,\ldots,k,\ j=2,3$, with color $j$; color $u_{2i-1}^3u_{2i}^3,\ i=1,\ldots,k/2$, with color $2$;
color $u_{2i-1}^2u_{2i}^2,\ i=1,\ldots,k/2$, with color $3$; and color the remaining edges with color $1$.
Thus $J_k$ is $3-$edge-colorable for even $k$.
\\

Now  consider the perfect pseudo-matching $M=\{v_iu_i^1,v_iu_i^2,v_iu_i^3\ :\ 1\le i\le k\}$ in the flower snark $J_k$, $k\ge 3$.
In fact, contract every claw induced by $\{u_i^1,u_i^2,u_i^3,v_i\}$ to a new vertex $v_i^{'}$, $1\le i\le k$.  Then  $J_k/M$ consists of three 
cycles $v_1^{'}v_2^{'}\ldots v_k^{'}v_1^{'}$, so  $J_k/M$  is planar.

  \item[(c)] {\bf Goldberg snarks}
     
     Goldberg~\cite{Goldberg} constructed an infinite family of snarks, $G_5,G_7,\ldots$.
For every odd $k\ge 5$, the vertex set of the {\sf Goldberg snark} $G_k$ satisfies
 $$V(G_k)=\{v_j^t\ :\ 1\le t\le k,\ 1\le j\le 8\}.$$ 
 The subgraph $B_t$ induced by $$\{v_1^t,v_2^t,\ldots,v_8^t\} \mbox{ and }
\{v_1^tv_2^t,v_1^tv_7^t,v_2^tv_8^t,v_3^tv_4^t,v_3^tv_8^t,v_4^tv_7^t,v_5^tv_6^t,v_6^tv_7^t,v_6^tv_8^t\}$$ is 
called a {\sf basic block.}
The Goldberg snark is constructed by joining each basic block $B_t$ with  $B_{t+1}$  by the  edges
$v_2^tv_1^{t+1},v_4^tv_3^{t+1},$ and  $v_5^tv_5^{t+1}$
 where 
the subscripts of basic blocks and the superscripts of vertices are read modulo $k$.

  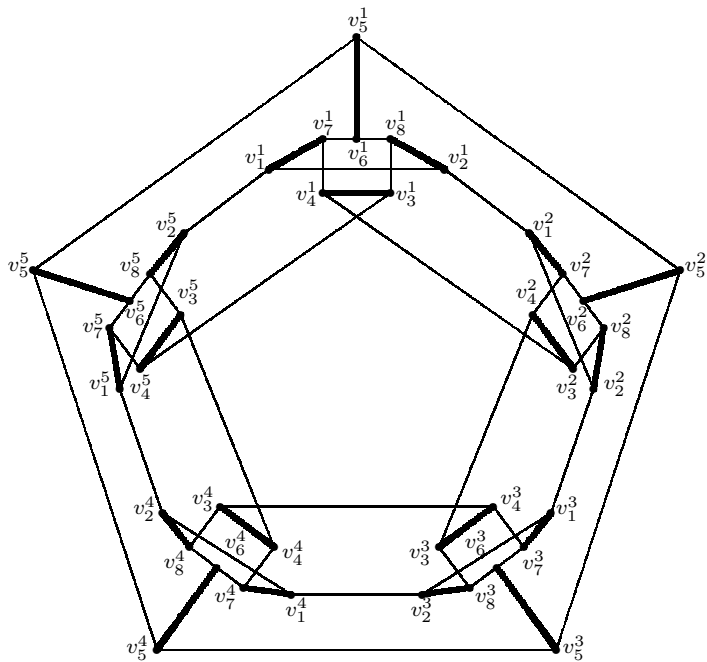
\begin{figure}[ht]

\setlength{\unitlength}{.09cm}
\begin{center}
\begin{picture}(70,80)
\put(0,0){\circle*{1.2}}
\put(-4.5,-.5){\scriptsize$v_5^4$}
\put(8.817,12.135){\circle*{1.2}}
\put(10,15){\scriptsize$v_6^4$}
\put(12.817,9.135){\circle*{1.2}}
\put(8.5,7){\scriptsize$v_7^4$}
\put(4.817,15.135){\circle*{1.2}}
\put(1,12){\scriptsize$v_8^4$}
\put(17.317,15.135){\circle*{1.2}}
\put(18.5,13.5){\scriptsize$v_4^4$}
\put(9.317,21.135){\circle*{1.2}}
\put(5.2,21.5){\scriptsize$v_3^4$}
\put(19.817,8.135){\circle*{1.2}}
\put(19,5){\scriptsize$v_1^4$}
\put(0.817,20.135){\circle*{1.2}}
\put(-3.5,20){\scriptsize$v_2^4$}

\put(59,0){\circle*{1.2}}
\put(60,-.5){\scriptsize$v_5^3$}
\put(50.183,12.135){\circle*{1.2}}
\put(45.5,15){\scriptsize$v_6^3$}
\put(46.183,9.135){\circle*{1.2}}
\put(47,7){\scriptsize$v_8^3$}
\put(54.183,15.135){\circle*{1.2}}
\put(54,12){\scriptsize$v_7^3$}
\put(41.683,15.135){\circle*{1.2}}
\put(37.25,13.5){\scriptsize$v_3^3$}
\put(49.683,21.135){\circle*{1.2}}
\put(50.75,21.5){\scriptsize$v_4^3$}
\put(39.183,8.135){\circle*{1.2}}
\put(37.5,5){\scriptsize$v_2^3$}
\put(58.183,20.135){\circle*{1.2}}
\put(59,20){\scriptsize$v_1^3$}

\put(77.3,56.05){\circle*{1.2}}
\put(78,56){\scriptsize$v_5^2$}
\put(63,51.52){\circle*{1.2}}
\put(60.5,48.){\scriptsize$v_6^2$}
\put(66,47.52){\circle*{1.2}}
\put(67,47){\scriptsize$v_8^2$}
\put(60,55.52){\circle*{1.2}}
\put(61,56){\scriptsize$v_7^2$}
\put(61.5,41.52){\circle*{1.2}}
\put(59,38){\scriptsize$v_3^2$}
\put(55.5,49.52){\circle*{1.2}}
\put(53,52.){\scriptsize$v_4^2$}
\put(64.5,38.52){\circle*{1.2}}
\put(66,38.5){\scriptsize$v_2^2$}
\put(55,61.52){\circle*{1.2}}
\put(55.5,62){\scriptsize$v_1^2$}

\put(29.5,90.5){\circle*{1.2}}
\put(28.2,92.2){\scriptsize$v_5^1$}
\put(29.5,75.5){\circle*{1.2}}
\put(28.2,72.5){\scriptsize$v_6^1$}
\put(24.5,75.5){\circle*{1.2}}
\put(23,77){\scriptsize$v_7^1$}
\put(34.5,75.5){\circle*{1.2}}
\put(33.75,77){\scriptsize$v_8^1$}
\put(24.5,67.5){\circle*{1.2}}
\put(20.5,66.5){\scriptsize$v_4^1$}
\put(34.5,67.5){\circle*{1.2}}
\put(35.3,66.5){\scriptsize$v_3^1$}

\put(42.5,71){\circle*{1.2}}
\put(43,72){\scriptsize$v_2^1$}
\put(16.5,71){\circle*{1.2}}
\put(13.,72){\scriptsize$v_1^1$}

\put(-18.3,56.05){\circle*{1.2}}
\put(-22,56){\scriptsize$v_5^5$}
\put(-4,51.52){\circle*{1.2}}
\put(-4.8,49){\scriptsize$v_6^5$}
\put(-7,47.52){\circle*{1.2}}
\put(-11,47){\scriptsize$v_7^5$}
\put(-1,55.52){\circle*{1.2}}
\put(-5.5,56){\scriptsize$v_8^5$}
\put(-2.5,41.52){\circle*{1.2}}
\put(-4,38){\scriptsize$v_4^5$}
\put(3.5,49.52){\circle*{1.2}}
\put(3,52.){\scriptsize$v_3^5$}
\put(4,61.52){\circle*{1.2}}
\put(-.20,62){\scriptsize$v_2^5$}
\put(-5.5,38.52){\circle*{1.2}}
\put(-10.,38.5){\scriptsize$v_1^5$}

\thicklines
\linethickness{0.15mm}

\qbezier(0,0)(0,0)(59,0)
\qbezier(0,0)(0,0)(-18.3,56.05)
\qbezier(77.3,56.05)(77.3,56.05)(59,0)
\qbezier(77.3,56.05)(77.3,56.05)(29.5,90.5)
\qbezier(-18.3,56.05)(-18.3,56.05)(29.5,90.5)

\qbezier(12.817,9.135)(12.817,9.135)(4.817,15.135)
\qbezier(17.317,15.135)(17.317,15.135)(12.817,9.135)
\qbezier(9.317,21.135)(4.817,15.135)(4.817,15.135)

\qbezier(46.183,9.135)(46.183,9.135)(54.183,15.135)
\qbezier(46.183,9.135)(43.183,13.135)(41.683,15.135)
\qbezier(54.183,15.135)(54.183,15.135)(49.683,21.135)

\qbezier(66,47.52)(66,47.52)(60,55.52)
\qbezier(66,47.52)(61.5,41.52)(61.5,41.52)
\qbezier(60,55.52)(60,55.52)(55.5,49.52)

\qbezier(24.5,75.5)(24.5,75.5)(34.5,75.5)
\qbezier(24.5,67.5)(24.5,70.5)(24.5,75.5)
\qbezier(34.5,67.5)(34.5,70.5)(34.5,75.5)

\qbezier(-7,47.52)(-7,47.52)(-1,55.52)
\qbezier(-7,47.52)(-2.5,41.52)(-2.5,41.52)
\qbezier(-1,55.52)(-1,55.52)(3.5,49.52)

\qbezier(34.5,67.5)(-2.5,41.52)(-2.5,41.52)
\qbezier(17.317,15.135)(3.5,49.52)(3.5,49.52)
\qbezier(24.5,67.5)(24.5,67.5)(61.5,41.52)
\qbezier(41.683,15.135)(55.5,49.52)(55.5,49.52)
\qbezier(9.317,21.135)(49.683,21.135)(49.683,21.135)

\qbezier(39.183,8.135)(19.817,8.135)(19.817,8.135)

\qbezier(58.183,20.135)(39.183,8.135)(39.183,8.135)
\qbezier(19.817,8.135)(0.817,20.135)(0.817,20.135)
\qbezier(64.5,38.52)(55,61.52)(55,61.52)
\qbezier(42.5,71)(16.5,71)(16.5,71)
\qbezier(4,61.52)(-5.5,38.52)(-5.5,38.52)

\qbezier(64.5,38.52)(64.5,38.52)(58.183,20.135)
\qbezier(-5.5,38.52)(-5.5,38.52)(0.817,20.135)
\qbezier(42.5,71)(42.5,71)(55,61.52)
\qbezier(4,61.52)(4,61.52)(16.5,71)

\thicklines
\linethickness{0.6mm}
\qbezier(0,0)(0,0)(8.817,12.135)
\qbezier(59,0)(59,0)(50.183,12.135)
\qbezier(29.5,75.5)(29.5,75.5)(29.5,90.5)
\qbezier(-18.3,56.05)(-18.3,56.05)(-4,51.52)
\qbezier(77.3,56.05)(77.3,56.05)(63,51.52)

\qbezier(17.317,15.135)(17.317,15.135)(9.317,21.135)

\qbezier(41.683,15.135)(41.683,15.135)(49.683,21.135)

\qbezier(61.5,41.52)(61.5,41.52)(55.5,49.52)

\qbezier(24.5,67.5)(24.5,67.5)(34.5,67.5)

\qbezier(-2.5,41.52)(-2.5,41.52)(3.5,49.52)

\qbezier(12.817,9.135)(19.817,8.135)(19.817,8.135)
\qbezier(4.817,15.135)(0.817,20.135)(0.817,20.135)
\qbezier(58.183,20.135)(58.183,20.135)(54.183,15.135)
\qbezier(39.183,8.135)(39.183,8.135)(46.183,9.135)
\qbezier(64.5,38.52)(64.5,38.52)(66,47.52)
\qbezier(55,61.52)(55,61.52)(60,55.52)
\qbezier(34.5,75.5)(42.5,71)(42.5,71)
\qbezier(16.5,71)(16.5,71)(24.5,75.5)

\qbezier(-1,55.52)(4,61.52)(4,61.52)
\qbezier(-5.5,38.52)(-5.5,38.52)(-7,47.52)

\end{picture}


\end{center}
\caption{\small\it 
The Goldberg snark $G_5$ and a planarizing perfect matching $M$ visualized by bold-face edges.}
\label{FIG:Goldberg5}
\end{figure}

Now  consider the perfect matching $M=\{v_1^tv_7^t,v_2^tv_8^t,v_3^tv_4^t,v_5^tv_6^t\ :\ 1\le t\le k\}$ in the Goldberg snark $G_k$, $k\ge 5$. 

By such a drawing of  $G_k$ as exemplified by Figure~\ref{FIG:Goldberg5}  and  by Theorem~\ref{Th:drawing} below,  $G_k/M$ is   planar.

\item[(d)] {\bf Celmins-Swart snarks}
 
  Planarizing perfect pseudo-matchings for the two Celmins-Swart  snarks are shown in  Figure~\ref{FIG:Celmins-Swart}.
 \begin{figure}[ht]

\setlength{\unitlength}{0.125cm}
\vspace{1cm}
\begin{center}


\begin{picture}(102,20)

\put(20,0){\circle*{1.2}}
\put(25,0){\circle*{1.2}}
\put(20,5){\circle*{1.2}}
\put(25,5){\circle*{1.2}}
\put(7.5,5){\circle*{1.2}}
\put(12.5,5){\circle*{1.2}}
\put(32.5,5){\circle*{1.2}}
\put(37.5,5){\circle*{1.2}}
\put(7.5,10){\circle*{1.2}}
\put(12.5,10){\circle*{1.2}}
\put(32.5,10){\circle*{1.2}}
\put(37.5,10){\circle*{1.2}}
\put(22.5,10){\circle*{1.2}}
\put(10,15){\circle*{1.2}}
\put(22.5,15){\circle*{1.2}}
\put(35,15){\circle*{1.2}}
\put(35,20){\circle*{1.2}}
\put(10,20){\circle*{1.2}}
\put(15,20){\circle*{1.2}}
\put(30,20){\circle*{1.2}}
\put(5,20){\circle*{1.2}}
\put(40,20){\circle*{1.2}}
\put(0,25){\circle*{1.2}}
\put(45,25){\circle*{1.2}}
\put(15,25){\circle*{1.2}}
\put(30,25){\circle*{1.2}}

\put(64,25){\circle*{1.2}}
\put(96,25){\circle*{1.2}}
\put(80,25){\circle*{1.2}}
\put(80,21){\circle*{1.2}}
\put(92,21){\circle*{1.2}}
\put(76,21){\circle*{1.2}}
\put(84,21){\circle*{1.2}}
\put(68,21){\circle*{1.2}}
\put(76,17){\circle*{1.2}}
\put(80,17){\circle*{1.2}}
\put(84,17){\circle*{1.2}}
\put(80,14){\circle*{1.2}}
\put(82,7){\circle*{1.2}}
\put(78,7){\circle*{1.2}}
\put(75,9){\circle*{1.2}}
\put(85,9){\circle*{1.2}}
\put(72,8){\circle*{1.2}}
\put(88,8){\circle*{1.2}}
\put(70,4){\circle*{1.2}}
\put(68,6){\circle*{1.2}}
\put(64,4){\circle*{1.2}}
\put(68,0){\circle*{1.2}}
\put(96,4){\circle*{1.2}}
\put(92,0){\circle*{1.2}}
\put(90,4){\circle*{1.2}}
\put(92,6){\circle*{1.2}}

\thicklines
\linethickness{0.15mm}

\qbezier(20,0)(20,0)(25,5)
\qbezier(20,5)(20,5)(25,0)
\qbezier(7.5,5)(7.5,5)(12.5,10)
\qbezier(7.5,10)(7.5,10)(12.5,5)
\qbezier(37.5,5)(37.5,5)(32.5,10)
\qbezier(37.5,10)(37.5,10)(32.5,5)
\qbezier(12.5,10)(20,5)(20,5)
\qbezier(32.5,10)(25,5)(25,5)
\qbezier(5,20)(0,12.5)(7.5,5)
\qbezier(0,25)(-2,17.5)(7.5,10)
\qbezier(40,20)(45,12.5)(37.5,5)
\qbezier(45,25)(47,17.5)(37.5,10)
\qbezier(12.5,5)(12.5,5)(20,0)
\qbezier(32.5,5)(32.5,5)(25,0)
\qbezier(0,25)(0,25)(45,25)
\qbezier(5,20)(5,20)(15,20)
\qbezier(30,20)(30,20)(40,20)
\qbezier(30,20)(30,20)(22.5,15)
\qbezier(15,20)(15,20)(22.5,15)

\qbezier(64,25)(64,25)(96,25)
\qbezier(68,21)(68,21)(76,21)
\qbezier(64,25)(64,25)(68,6)
\qbezier(68,21)(68,21)(64,4)
\qbezier(84,21)(84,21)(92,21)
\qbezier(96,25)(96,25)(92,6)
\qbezier(92,21)(92,21)(96,4)
\qbezier(68,0)(68,0)(92,0)
\qbezier(64,4)(64,4)(96,4)
\qbezier(68,0)(68,0)(68,6)
\qbezier(92,0)(92,0)(92,6)

\qbezier(76,21)(76,21)(76,17)
\qbezier(84,21)(84,21)(84,17)

\qbezier(82,7)(82,7)(76,17)
\qbezier(78,7)(78,7)(84,17)

\qbezier(80,14)(80,14)(75,9)
\qbezier(80,14)(80,14)(85,9)

\qbezier(78,7)(75,9)(75,9)
\qbezier(82,7)(82,7)(85,9)

\thicklines
\linethickness{0.6mm}
\qbezier(20,0)(20,0)(25,0)
\qbezier(7.5,5)(7.5,5)(12.5,5)
\qbezier(37.5,5)(37.5,5)(32.5,5)
\qbezier(7.5,10)(7.5,10)(10,15)
\qbezier(37.5,10)(37.5,10)(35,15)
\qbezier(12.5,10)(12.5,10)(10,15)
\qbezier(32.5,10)(32.5,10)(35,15)
\qbezier(10,15)(10,15)(10,20)
\qbezier(35,15)(35,15)(35,20)
\qbezier(20,5)(20,5)(22.5,10)
\qbezier(25,5)(25,5)(22.5,10)
\qbezier(22.5,10)(22.5,10)(22.5,15)
\qbezier(5,20)(5,20)(0,25)
\qbezier(40,20)(40,20)(45,25)
\qbezier(15,20)(15,20)(15,25)
\qbezier(30,20)(30,20)(30,25)

\qbezier(80,25)(80,25)(80,21)
\qbezier(76,21)(76,21)(84,21)
\qbezier(64,25)(64,25)(68,21)
\qbezier(96,25)(96,25)(92,21)
\qbezier(76,17)(76,17)(84,17)
\qbezier(80,17)(80,17)(80,14)
\qbezier(78,7)(78,7)(82,7)
\qbezier(64,4)(64,4)(68,0)
\qbezier(96,4)(96,4)(92,0)
\qbezier(90,4)(90,4)(88,8)
\qbezier(92,6)(92,6)(88,8)
\qbezier(88,8)(88,8)(85,9)
\qbezier(70,4)(70,4)(72,8)
\qbezier(68,6)(68,6)(72,8)
\qbezier(72,8)(72,8)(75,9)

\end{picture}


\end{center}
\caption{\small\it 
The  Celmins-Swart  snarks and their planarizing perfect pseudo-matchings exhibited by the bold-face edges
$($cf. Theorem~$\ref{Th:drawing}$ below$)$.}
\label{FIG:Celmins-Swart}
\end{figure}
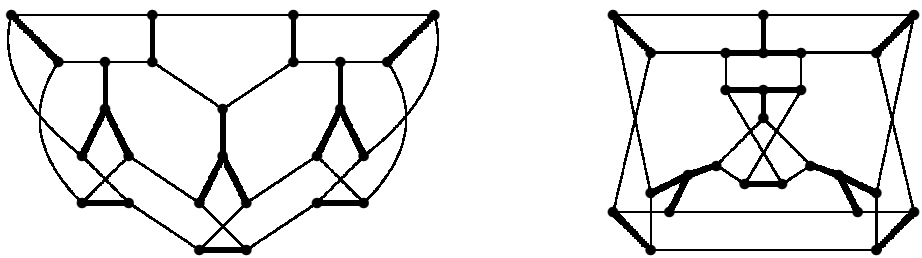
  
   \end{description}

 We also examined other snarks and determined planarizing perfect pseudo-matchings. 
However, we cannot conclude that every snark has a planarizing perfect pseudo-matching (otherwise we had an easy proof of the CDC Conjecture).

\begin{example}
By using computer programming we found a
  snark of order $26$ shown in  Figure~$\ref{FIG:Snark26}$  which has  no planarizing perfect pseudo-matching but it has a {\rm SUD}-$K_5-$minor-free perfect pseudo-matching.

  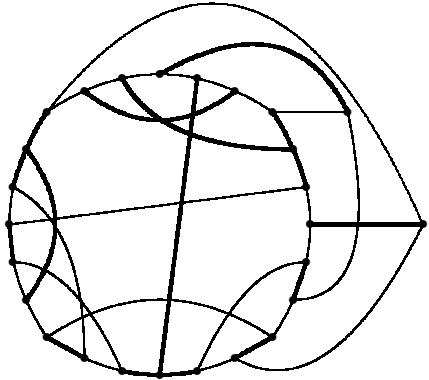
\begin{figure}[ht]

\setlength{\unitlength}{.1cm}
\vspace{1cm}
\begin{center}


\begin{picture}(50,40)

\put(0,0){\circle*{1.2}}

 \put(5,-2.75){\circle*{1.2}}
 \put(10,-4.5){\circle*{1.2}}
 \put(15,-5){\circle*{1.2}}
 \put(20,-4.5){\circle*{1.2}}
 \put(25,-2.75){\circle*{1.2}}

\put(30,0){\circle*{1.2}}

 \put(32.75,5){\circle*{1.2}}
 \put(34.5,10){\circle*{1.2}}
 \put(35,15){\circle*{1.2}}
 \put(34.5,20){\circle*{1.2}}
 \put(32.75,25){\circle*{1.2}}

\put(30,30){\circle*{1.2}}

 \put(5,32.75){\circle*{1.2}}
 \put(10,34.5){\circle*{1.2}}
 \put(15,35){\circle*{1.2}}
 \put(20,34.5){\circle*{1.2}}
 \put(25,32.75){\circle*{1.2}}

\put(0,30){\circle*{1.2}}

\put(-2.75,5){\circle*{1.2}}
 \put(-4.5,10){\circle*{1.2}}
 \put(-5,15){\circle*{1.2}}
 \put(-4.5,20){\circle*{1.2}}
 \put(-2.75,25){\circle*{1.2}}
 
  \put(50,15){\circle*{1.2}}
 
 \put(40,30){\circle*{1.2}}

\thicklines
\linethickness{0.15mm}


\qbezier(40,30)(30,30)(30,30)
\qbezier(40,30)(45,5)(32.75,5)
\qbezier(50,15)(28,65)(0,30)
\qbezier(50,15)(37.5,-10)(25,-2.75)

\qbezier(0,30)(15,40)(30,30)
\qbezier(0,30)(-10,15)(0,0)
\qbezier(0,0)(15,-10)(30,0)
\qbezier(30,0)(40,15)(30,30)

\qbezier(-5,15)(-5,15)(34.5,20)

\qbezier(-4.5,20)(5,15)(5,-2.75)
\qbezier(-4.5,10)(3.5,10)(10,-4.5)

\qbezier(0,0)(15,10)(30,0)

\qbezier(34.5,10)(26.5,10)(20,-4.5)
\thicklines
\linethickness{0.4mm}
\qbezier(15,-5)(15,-5)(20,34.5)
\qbezier(0,0)(0,0)(5,-2.75)
\qbezier(10,-4.5)(15,-5.5)(20,-4.5)
\qbezier(30,0)(30,0)(25,-2.75)
\qbezier(32.75,5)(32.75,5)(34.5,10)
\qbezier(34.5,20)(33,25.5)(30,30)
\qbezier(10,34.5)(15,25)(32.75,25)
\qbezier(5,32.75)(15,25)(25,32.75)
\qbezier(40,30)(32.5,45)(15,35)
\qbezier(0,30)(-3,25.5)(-4.5,20)
\qbezier(-2.75,5)(5,15)(-2.75,25)
\qbezier(50,15)(50,15)(35,15)
\qbezier(-5,15)(-5,15)(-4.5,10)
\end{picture}


\end{center}
\vspace{1cm}
\caption{\small\it 
An snark of order $26$ without any planarizing perfect pseudo-matchings and its  {\rm SUD}-$K_5-$minor-free perfect pseudo-matching visualized by bold-face edges.}
\label{FIG:Snark26}
\end{figure}

\end{example}  
 
 In fact, for all snarks with up to $32$ vertices which contain a stable dominating cycle it was checked whether they contain a planarizing / $K_5-$minor-free / SUD-$K_5-$minor-free perfect pseudo-matching (the perfect pseudo-matching does not have to be the complement of the stable dominating cycle). It was determined that \\
 
   {\it  there are $4615$ snarks $G_3$ with up to $32$ vertices which have a stable      dominating cycle. $4612$  of them contain a planarizing perfect pseudo-matching, $2$ have no planarizing perfect pseudo-matching, but a $K_5-$minor-free perfect pseudo-matching, and one has no $K_5-$minor-free perfect pseudo-matching but a {\rm SUD}-$K_5-$minor-free perfect pseudo-matching. So all $4615$ snarks have a 
{\rm SUD}-$K_5-$minor-free perfect pseudo-matching $M$ $($which in turn guarantees that $G_3/M$ has a  compatible cycle decomposition$)$.}\\

Moreover, we checked also directly the complements of the stable dominating cycles and found the following. 
 Surprisingly, for all the stable dominating cycles in all snarks with up to $32$ vertices the complements where never (!) planarizing. Furthermore, there is also no stable dominating cycle whose complement is a $K_5-$minor-free perfect pseudo-matching. Out of the $4615$ snarks with stable dominating cycles there are $3045$ snarks which contain a stable dominating cycle whose complement is a {\rm SUD}-$K_5-$minor-free perfect pseudo-matching. For the other $1570$ snarks the complements of all stable dominating cycles are a {\rm SUD}-$K_5-$minor perfect pseudo-matching, but nevertheless the contraction of these perfect pseudo-matchings have a compatible cycle decomposition. So in this case we see that the complement of all stable dominating cycles is at least a perfect pseudo-matching after whose contraction the eulerian graph has a cycle decomposition compatible with the eulerian trail corresponding to the given dominating cycle in the snark. 
\\
   
Moreover, $3-$edge-colorability and compatible cycle decompositions can be related; see our next result.

\begin{theorem}\label{Th:3edgecoloring}
Let $G$ be a cubic graph and let $M$ be a perfect pseudo-matching in $G$. Let $\mathcal{T}$ be the transition system in $G/M$  defined by pairs of adjacent edges in $G_0= G\setminus E(M)$.  For a compatible cycle decomposition $\mathcal{S}$ of $(G/M,\mathcal{T})$ define the intersection graph $\mathcal{I(S)}$ whose vertices correspond to the elements of $\mathcal{S}$ and $xy$ is in $E(\mathcal{I(S)})$ if and only if the corresponding cycles $C_x$ and $C_y$ in $\mathcal{S}$ have at least one vertex in common. The following  {\rm \bf (i)} and
{\rm \bf (ii)} are equivalent.

\begin{description}
\item[(i)] $G$ is $3-$edge-colorable.

\item[(ii)] $\chi(\mathcal{I(S)}) \leq 3$ for at least one compatible cycle decomposition $\mathcal{S}$ of $(G/M, \mathcal{T})$.

Furthermore,

\item[(iii)] $\chi(\mathcal{I(S)}) = 2$ if  and only if $M$ is a perfect matching and represents a color class in a $3-$edge-coloring of $G$.
\end{description}
\end{theorem}

\begin{proof}
Let $G$ be a cubic graph with a proper edge coloring $c$ with color set $\{1,2,3\}$.
The coloring $c$ induces a color on every edge of $G/M$. Since $G$ is $3-$regular and $c$ is  proper, every vertex of $G/M$
is adjacent to $0$ or $2$ edges of color $i$, $i=1,2,3$. Therefore, the subgraph induced by $i-$colored edges in $G/M$ is a disjoint union of 
$i-$colored cycles, $i=1,2,3$; moreover, each such $i-$colored cycle is a compatible cycle. Let $\mathcal{S}$ be the set of all such compatible $i-$colored cycles for $i=1,2,3$. Now,  give color $i$ to every vertex of $\mathcal{I(S)}$ which corresponds to a compatible $i-$colored cycle in $\mathcal{S}$, $i=1,2,3$. Thus   $\mathcal{I(S)}$ has a proper vertex coloring with color set $\{1,2,3\}$.

By reversing the argument the implication {\bf (ii)}$\rightarrow${\bf (i)} follows easily. The remainder of the proof is easily established.
\end{proof}

 A planarizing perfect pseudo-matching $M$ of the Petersen graph is shown in Figure~\ref{FIG:Petersen}, which in turn served as the basis for constructing  cubic graphs $G$ with a stable dominating cycle $C$ with a planarizing perfect pseudo-matching $M$ having only $2$ components which are $K_{1,3}$. This in turn led to a simple uniquely hamiltonian graph of minimum degree $4$,~\cite{Fleischner2014}.
 We note in passing that any stable dominating cycle in a cubic graph can be used as the basis for constructing a simple uniquely hamiltonian 
 graph of minimum degree $4$.

In view of our remarks preceding Theorem~\ref{Th:3edgecoloring}, we are led to the    following question.

\begin{question}\label{Q:1}
  Given a cubic graph $G_3$ with a stable dominating cycle $C$ and corresponding perfect pseudo-matching $M=E(G_3)\setminus E(C)$. Is it true that $(G_3 / M,\mathcal{T}_M)$ has a compatible cycle decomposition?
\end{question}

\begin{example}
By using computer programming we found a
  snark $G_3$ of order $28$ shown in  Figure~$\ref{FIG:Snark28}$  which has exactly   one stable dominating cycle $C$.
  The perfect pseudo-matching $M=E(G_3)\setminus E(C)$ is  visualized by bold-face edges. $G_3/M$ contains a {\rm SUD}-$K_5-$minor.

  \begin{figure}[ht]

\setlength{\unitlength}{.1cm}
\vspace{1cm}
\begin{center}


\begin{picture}(35,40)

\put(15,-5){\circle*{1.2}}
 
\put(19.5,-4.5){\circle*{1.2}}
\put(10.5,-4.5){\circle*{1.2}} 
 
\put(6,-3.5){\circle*{1.2}}
\put(24,-3.5){\circle*{1.2}} 

\put(0.6,-1.5){\circle*{1.2}}
\put(29.4,-1.5){\circle*{1.2}}

\put(32.5,3){\circle*{1.2}}
\put(-2.5,3){\circle*{1.2}}
 
\put(34.5,7.5){\circle*{1.2}}
\put(-4.5,7.5){\circle*{1.2}}
 
\put(35.5,12.5){\circle*{1.2}}
\put(-5.5,12.55){\circle*{1.2}}

\put(35.5,17.5){\circle*{1.2}}
\put(-5.5,17.55){\circle*{1.2}}

\put(34.5,22.5){\circle*{1.2}}
\put(-4.5,22.5){\circle*{1.2}}

\put(32.5,27){\circle*{1.2}}
\put(-2.5,27){\circle*{1.2}}

\put(29.4,31.5){\circle*{1.2}}
\put(.6,31.5){\circle*{1.2}}

\put(40,30){\circle*{1.2}}
\put(-10,30){\circle*{1.2}}

\put(6,33.5){\circle*{1.2}}
\put(24,33.5){\circle*{1.2}}

 \put(10.5,34.5){\circle*{1.2}}
 \put(19.5,34.5){\circle*{1.2}}

 \put(15,35){\circle*{1.2}}

\thicklines
\linethickness{0.15mm}


\qbezier(.6,31.5)(15,38.25)(29.4,31.5)
\qbezier(.6,-1.5)(15,-8.25)(29.4,-1.5)
\qbezier(.6,31.5)(-12,15)(.6,-1.5)
\qbezier(29.4,-1.5)(42,15)(29.4,31.5)
\thicklines
\linethickness{0.4mm}
\qbezier(15,35)(25,22)(34.5,22.5)
\qbezier(29.4,31.5)(29.4,31.5)(-4.5,22.5)
\qbezier(-5.5,12.5)(-5.5,12.5)(24,-3.5)
\qbezier(-2.5,3)(-2.5,3)(35.5,12.5)
\qbezier(24,33.5)(30,18)(35.5,17.5)
\qbezier(6,33.5)(0,18)(-5.5,17.5)
\qbezier(.6,31.5)(1.5,15)(-4.5,7.5)
\qbezier(15,-5)(15,-5)(34.5,7.5)
\qbezier(6,-3.5)(18,1.5)(29.4,-1.5)
\qbezier(0.6,-1.5)(9,1.5)(19.5,-4.5)
\qbezier(40,30)(33,40)(19.5,34.5)
\qbezier(40,30)(45,15)(32.5,3)
\qbezier(40,30)(40,30)(32.5,27)
\qbezier(-10,30)(0,40)(10.5,34.5)
\qbezier(-10,30)(-10,30)(-2.5,27)
\qbezier(-10,30)(-9,-12)(10.5,-4.5)
\end{picture}


\end{center}
\vspace{1cm}
\caption{\small\it 
An snark of order $28$ with precisely one stable dominating cycle $C$.}
\label{FIG:Snark28}
\end{figure}
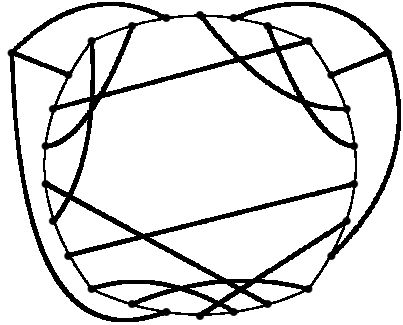

\end{example}

\begin{proposition}\label{Prop:1}
Suppose Question~$\ref{Q:1}$ has a positive answer. Then
 \begin{description}
\item[(a)] the {\rm SC Conjecture} is true;

\item[(b)] the {\rm CDC Conjecture} can be reduced to the {\rm DC Conjecture}.
\end{description} 
\end{proposition}

\begin{proof}
 Suppose Sabidussi’s Compatibility Conjecture (SC Conjecture) is false. Then it is false even for an eulerian graph $G$ with eulerian trail $T_\epsilon$ and  $4\le \delta(G)\le \Delta(G)\le 6$. Let $G_3$ be the cubic graph and $C_\epsilon$ the dominating cycle in $G_3$ corresponding to $G$ and $T_\epsilon$, respectively, as mentioned in the introduction following the statement of the SC Conjecture. We distinguish between two cases.

\begin{description}
\item[1)] $C_\epsilon$ is a stable cycle. Then $E(G_3) \setminus E(C_\epsilon)$ is a perfect pseudo-matching $M$, and by the supposition of this proposition $G = G_3/M$
has a cycle decomposition compatible with $\mathcal{T}_\epsilon$ (corresponding to $C_\epsilon$), contrary to our supposition that SC Conjecture is false for $G$  with eulerian trail $T_\epsilon$.\\

\item[2)] $C_\epsilon$ is not a stable cycle. Then there is a dominating cycle $C_1$ in $G_3$ with $V(C_\epsilon)\subseteq V(C_1)$. Considering $G = G_3/M$ and $T_\epsilon$ as in case 1), it follows
that $C_1$ corresponds to a spanning trail $T_{1}$ in $G$. Moreover,  $G \setminus E(T_{1})$ is a set of totally disjoint compatible cycles $\mathcal{S}^{(1)}$ in $(G,\mathcal{T})$.

Let $G^{(1)}$ be the subgraph of $G$ induced by $E(T_{1})$.
In $G^{(1)}$, $T_1$ is an eulerian trail. Suppressing in $G^{(1)}$ the  vertices of degree $2$ we transform $(G^{(1)},T_1)$ into $(G^{'},T^{'})$ with $G^{'}$ having  vertices of degree $4$ and $6$ only, and $T^{'}$ being an eulerian trail of $G^{'}$. Let $G^{'}_3$ be the cubic graph corresponding to $G^{'}$, and let $C^{'}$ be the dominating cycle of $G^{'}_3$ corresponding to $T^{'}$ and thus to $C_1$. 

\end{description}

Now we consider the two cases 1) and 2) with $(G^{'}_3,C^{'})$ in place of $(G_3,C)$; and so on. Ultimately, for some $j > 0,\ G^{(j)}$ is nothing but a cycle $C^{*}$,
or it has a cycle decomposition $\mathcal{S}^*$ compatible with the eulerian trail $T^{(j)}$ of $G^{(j)}$ because the corresponding dominating cycle $C^{(j)}$ of $G_3^{(j)}$ is stable. 
 Thus,  $\mathcal{S}^{(1)}\cup \ldots \cup \mathcal{S}^{(j-1)}\cup\{C^{*}\}$ or  $\mathcal{S}^{(1)}\cup \ldots \cup \mathcal{S}^{(j-1)}\cup \mathcal{S}^*$, respectively, corresponds to a cycle decomposition  $\mathcal{S}$ of $G$  compatible with $T_\epsilon$ after step-by-step inserting anew the suppressed  vertices of degree $2$. The validity of SC Conjecture for $G$ follows, contrary to the original supposition.\\

Moreover, since it suffices to consider snarks when dealing with the CDC  Conjecture we may first consider the DC  Conjecture. Thus for a given snark $G_3$, if we can find a dominating cycle $C$, then we can construct $(G,T)$ by contracting the perfect pseudo-matching 
$M = G_3 \setminus E(C)$ with $G = G_3/M$ and $T$ corresponding to $C$. Now we argue algorithmically as above to
obtain a cycle decomposition  $\mathcal{S}$  of $G$ compatible with $T$. Clearly, $\mathcal{S}$ corresponds to a
set $\mathcal{S}_3$ of cycles in $G_3$ covering the edges of $M$ twice and the edges of $C$ once.
Thus $\mathcal{S}_3 \cup \{C\}$ is a cycle double cover of $G_3$ containing the dominating cycle 
$C$.   Then the CDC Conjecture is true if the DC Conjecture is true.                                                                                                                                                      
\end{proof}

 However, we also made a computer search establishing the usefulness of perfect pseudo-matchings - we include  Table~\ref{TBL:snarks} produced by the third author  and discuss the advantage of perfect pseudo-matchings vis-a-vis perfect matchings.

 \setlength{\tabcolsep}{2pt}
  \begin{table}[tp!]
    \centering
    \begin{tabular}{rrrrrrrr}
      \toprule[1.5pt]
    \multicolumn{8}{c}{$\text{Snarks of order $\le 32$}$}  \\
      \cmidrule(lr){1-8}
      \footnotesize$n$&\footnotesize$s(n)$&\footnotesize$\overline{sppm}(n)$&\footnotesize$\overline{spppm}(n)$&\footnotesize$\overline{spmK_5}(n)$&\footnotesize$\overline{sppmK_5}(n)$&\footnotesize$\overline{spmSK_5}(n)$&\footnotesize$\overline{sppmSK_5}(n)$  \\
      \midrule
      \footnotesize$\mathbf{10}$& \footnotesize$1$ &\footnotesize$\mathbf{1}$&\footnotesize$0$&\footnotesize$\mathbf{1}$&\footnotesize$0$&\footnotesize$\mathbf{1}$&\footnotesize$0$\\
\footnotesize$\mathbf{18}$&\footnotesize $2$   & \footnotesize$\mathbf{1}$ &\footnotesize$0$&\footnotesize$\mathbf{1}$&\footnotesize$0$&\footnotesize$\mathbf{1}$&\footnotesize$0$\\
\footnotesize$\mathbf{20}$&\footnotesize$6$ & \footnotesize$\mathbf{5}$ &\footnotesize$0$&\footnotesize$\mathbf{5}$&\footnotesize$0$&\footnotesize$\mathbf{4}$&\footnotesize$0$\\
\footnotesize$\mathbf{22}$ &\footnotesize$31$ & \footnotesize$\mathbf{29}$&\footnotesize$0$&\footnotesize$\mathbf{29}$&\footnotesize$0$&\footnotesize$\mathbf{14}$&\footnotesize$0$\\
\footnotesize$\mathbf{24}$& \footnotesize$155$ &\footnotesize $\mathbf{146}$&\footnotesize$0$ &\footnotesize$\mathbf{146}$&\footnotesize$0$ &\footnotesize$\mathbf{97}$&\footnotesize$0$\\
\footnotesize$\mathbf{26}$& \footnotesize$1297$ & \footnotesize$\mathbf{1239}$  &\footnotesize$2$&\footnotesize$\mathbf{1239}$&\footnotesize$0$&\footnotesize$\mathbf{822}$&\footnotesize$0$ \\
\footnotesize$\mathbf{28}$&\footnotesize$12517$&\footnotesize$\mathbf{12102}$&\footnotesize$45$&\footnotesize$\mathbf{12102}$&\footnotesize$15$&\footnotesize$\mathbf{8374}$&\footnotesize$0$\\
\footnotesize$\mathbf{30}$&\footnotesize$139854$&\footnotesize$\mathbf{136850}$&\footnotesize$933$&\footnotesize$\mathbf{136850}$&\footnotesize$578$&\footnotesize$\mathbf{105321}$&\footnotesize$33$\\\footnotesize$\mathbf{32}$&\footnotesize$1764950$&\footnotesize$\mathbf{1740342}$&
\footnotesize$24268$&\footnotesize$\mathbf{1740342}$&\footnotesize$18537$&\footnotesize$\mathbf{1430228}$&\footnotesize$1062$\\
      \bottomrule[1.5pt]
    \end{tabular}
\caption{\cite{Klocker} \small\it 
\footnotesize{$s(n)=$ $\#$Snarks of order $n$; $\overline{sppm}(n)=$ $\#$Snarks of order $n$ with no planarizing perfect matching; $\overline{spppm}(n)=$ $\#$Snarks of order $n$ with no planarizing perfect pseudo-matching; $\overline{spmK_5}(n)=$ $\#$Snarks of order $n$ with no $K_5-$minor-free perfect matching; $\overline{sppmK_5}(n)=$ $\#$Snarks of order $n$ with no $K_5-$minor-free perfect pseudo-matching; $\overline{spmSK_5}(n)=$ $\#$Snarks of order $n$ with no $\rm SUD$-$K_5-$minor-free perfect matching; $\overline{sppmSK_5}(n)=$ $\#$Snarks of order $n$ with no $\rm SUD$-$K_5-$minor-free perfect pseudo-matching.}}
\label{TBL:snarks}
  \end{table}

In Table~\ref{TBL:snarks}, compare column $2i$ with column $2i-1$, $i=2,3,4$, in which a small portion  of snarks of order $n$ have some perfect matching with a property $X$ whereas almost all  snarks of order $n$ have some perfect pseudo-matching with the property $X$.
For example, out of the $1297$ snarks of order $26$, $58$ have some planarizing perfect matching while all snarks of order $26$ except two of them (one  is shown in Figure~\ref{FIG:Snark26}) have some planarizing perfect pseudo-matching; and all snarks of order $26$ have some $\rm SUD$-$K_5-$minor-free perfect pseudo-matching while $475$ of them have some $\rm SUD$-$K_5-$minor-free perfect matching.

 Nonetheless, snarks with planarizing perfect pseudo-matchings exist in abundance. To show this we need to define a certain drawing of a cubic graph $G$ with given perfect pseudo-matching $M$. 

\begin{lemma}\label{LEM:drawing}
Let a simple cubic graph $G$ with perfect pseudo-matching $M$ be given. Then we can draw $G$ in the plane in such a way that the crossings of the drawing involve only edges of $G \setminus E(M)$. Moreover, no three edges of $G$ cross each other in the same point of the plane.
\end{lemma}

\begin{proof}
The validity of the lemma rests on the following facts.

\begin{description}

\item[(i)] One can draw $M$ in the Euclidean plane $\mathbb{R}^2$ such that any two points of different components of $M$ are of distance $>1$ apart, say, and such that the drawing of $M$ does not contain any closed curve;

\item[(ii)] viewing $M$ as a point set, the set $\mathbb{R}^2\setminus M$  is connected;

\item[(iii)] every edge $e = xy \in E(G) \setminus E(M)$ can be drawn as a smooth curve in $\mathbb{R}^2$ such that all of $e$ except $x, y$ lie in $\mathbb{R}^2 \setminus M$;

\item[(iv)] a desired drawing of $G$ in $\mathbb{R}^2\setminus M$ can be achieved in a step-by-step manner. Namely, if we have already drawn a subgraph $G^\circ$ of $G$ with $G^\circ$ containing $M$ such that no three edges cross each other in the same point, and if $X^\circ$ is the set of crossing points of pairs of edges of $G^\circ$, then $\mathbb{R}^2\setminus (E(M) \cup X^\circ)$ is still a connected point set (and so the next edge can be drawn without passing through an element of $X^\circ$).

In addition we may achieve the property that

\item[(v)] no pair of edges in the drawing of $E(G)$ has more than one common point.  
\end{description}
\vspace{-1cm}
\end{proof}

  We call  a drawing having the property stated in Lemma~\ref{LEM:drawing}, a {\sf drawing with  $M$-avoiding intersections}. Note that if such a drawing exists, then there exists even one where adjacent edges do not intersect in $\mathbb{R}^2\setminus M$ (see (v) above).

Thus we may relate drawings with $M$-avoiding intersections to planarizing perfect pseudo-matchings $M$ as follows.

\begin{theorem}\label{Th:drawing}
Let $G$ be a simple cubic graph, and let $M$  be a perfect pseudo-matching. $M$ is planarizing if and only if there exists a drawing of $G$ with $M$-avoiding intersections such that if $e,f \in E(G) \setminus E(M)$ intersect, then  $e$ and $f$ are incident to different vertices of a component of  $M$.
\end{theorem}

\begin{proof}
Let $M$  be a planarizing perfect pseudo-matching of a cubic graph $G$. So there is a drawing of $G$ with  $M$-avoiding intersections
by Lemma~\ref{LEM:drawing} such that 
$G/M$ is planar. Such drawing of $G$ can be obtained by starting from a plane embedding of $G/M$ and by replacing the vertices 
of $G/M$ by the components of $M$. Thus $e$ and $f$ are adjacent to different vertices of a component of  $M$, for every edge-crossing
involving $\{e,f\}\subset E(G) \setminus E(M)$ in this drawing of $G$; otherwise, the edges  corresponding  to $e$ and $f$ in
$G/M$  intersect in said drawing.

Suppose conversely that $G$ has a drawing with $M$-avoiding intersections such that if $e,f \in E(G) \setminus E(M)$ intersect, then  $e$ and $f$ are adjacent to different vertices of a component of  $M$, but suppose that $G/M$ is not planar. So there exist some edge-crossing involving
$\{e^{'},f^{'}\}$ in any embedding of $G/M$. Thus, $e^{'}$ and $f^{'}$ can be assumed to be adjacent to four different vertices in $V(G/M)$. Since every vertex of $G/M$
corresponds to a component of $M$, two edges $e,f \in E(G) \setminus E(M)$ corresponding to $e^{'}$ and $f^{'}$ intersect but 
they are not adjacent to different vertices of a component of  $M$, which is a contradiction.
\end{proof}

\begin{lemma}\label{LEM:G*}
Let $G$ be a $2-$connected simple cubic graph having an embedding in the plane with precisely one edge-crossing $\{xx^{'''},x^{'}x^{''}\}\subset E(G)$. Define
                        $G_x^* := (G \setminus \{xx^{'''},x^{'}x^{''}\}) \cup B_0$
obtained by connecting $x$ with $a$ and $x^{'''}$ with $a^{'}$, $x^{'}$ with $b$ and $x^{''}$ with $b^{'}$, and where $B_0$ is the first Blanu\v{s}a block and $a,a^{'},b,b^{'}$ are half-edges of $B_0$ $($see Figures~$\ref{FIG:BlanusaBlocks}\ and~\ref{FIG:edge-crossing})$.  Also let $M$ 
be a  perfect pseudo-matching of $G \setminus \{xx^{'''},x^{'}x^{''}\}$. Then the following  is  true.
\begin{description}
\item[(i)]  $G_x^*$ is $3-$edge-colorable if and only if $G$ is $3-$edge-colorable. 
\item[(ii)] $G_x^*$ has a spanning subgraph homeomorphic to $G$.

\item[(iii)] $G_x^*$ has a planarizing perfect pseudo-matching $M_x^*$ with $M \subset M_x^*$.
\item[(iv)] Every {\rm CDC} of $G$ can be extended to a {\rm CDC} of $G_x^*$.
\end{description}  
 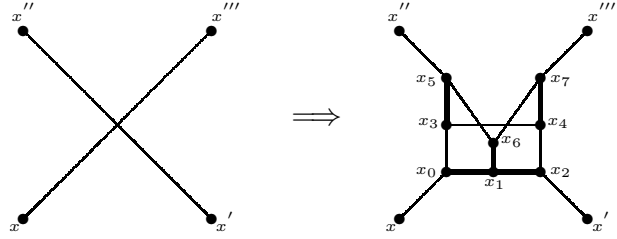
\begin{figure}[ht]

\setlength{\unitlength}{0.125cm}
\vspace{1cm}
\begin{center}


\begin{picture}(-15,20)
\put(10,2){\circle*{1.2}}
\put(6.75,1.75){\tiny$x_0$}
\put(10,7){\circle*{1.2}}
\put(7,7){\tiny$x_3$}
\put(10,12){\circle*{1.2}}
\put(6.75,11.5){\tiny$x_5$}
\put(15,2){\circle*{1.2}}
\put(14,0.5){\tiny$x_1$}
\put(15,5){\circle*{1.2}}
\put(15.8,5){\tiny$x_6$}
\put(20,2){\circle*{1.2}}
\put(21,1.75){\tiny$x_2$}
\put(20,7){\circle*{1.2}}
\put(20.75,7){\tiny$x_4$}
\put(20,12){\circle*{1.2}}
\put(21,11.5){\tiny$x_7$}

\put(25,17){\circle*{1.2}}
\put(25,18){\tiny$x^{'''}$}
\put(5,17){\circle*{1.2}}
\put(3.75,18){\tiny$x^{''}$}
\put(25,-3){\circle*{1.2}}
\put(25.5,-4.3){\tiny$x^{'}$}
\put(5,-3){\circle*{1.2}}
\put(3.5,-4.3){\tiny$x$}

\put(-15,17){\circle*{1.2}}
\put(-15,18){\tiny$x^{'''}$}
\put(-35,17){\circle*{1.2}}
\put(-36.25,18){\tiny$x^{''}$}
\put(-15,-3){\circle*{1.2}}
\put(-14.5,-4.3){\tiny$x^{'}$}
\put(-35,-3){\circle*{1.2}}
\put(-36.5,-4.3){\tiny$x$}

\put(-6.5,7){$\Longrightarrow$}

\thicklines
\linethickness{0.15mm}

\qbezier(-15,-3)(-15,-3)(-35,17)
\qbezier(-35,-3)(-35,-3)(-15,17)
\qbezier(10,2)(10,2)(20,2)
\qbezier(10,7)(10,7)(20,7)
\qbezier(10,12)(10,12)(15,5)
\qbezier(20,12)(20,12)(15,5)
\qbezier(10,2)(10,7)(10,12)
\qbezier(20,2)(20,7)(20,12)

\qbezier(5,-3)(5,-3)(10,2)
\qbezier(5,17)(5,17)(10,12)
\qbezier(25,17)(25,17)(20,12)
\qbezier(25,-3)(25,-3)(20,2)

\thicklines
\linethickness{0.6mm}
\qbezier(15,2)(15,5)(15,5)
\qbezier(10,2)(10,2)(20,2)
\qbezier(10,7)(10,7)(10,12)
\qbezier(20,7)(20,7)(20,12)

\end{picture}


\end{center}
\vspace{2mm}
\caption{\small\it 
The conversion of an edge-crossing to  a planarizing part.}
\label{FIG:edge-crossing}
\end{figure}

\end{lemma}
\begin{proof}
 Suppose $xx^{'''},x^{'}x^{''}\in E(G)$ are involved in the only edge-crossing in the given embedding of $G$, and suppose $M$ is a  perfect pseudo-matching of $G\setminus \{xx^{'''},x^{'}x^{''}\}$. 
Let $G_x^*$ be the graph, obtained from $G$ as described in the statement of the lemma and Figure~\ref{FIG:edge-crossing}.
\begin{description}
\item[(i)]
It is easy to check that the first Blanu\v{s}a block $B_0$ has precisely two types of proper $3-$edge-colorings:
 in one type  all half-edges $a,a^{'},b,b^{'}$ have the same color and in the second type half-edges $a$ and $a^{'}$ have the same color and 
 half-edges $b$ and $b^{'}$ also have the same color  different from the color of $a$ and $a^{'}$. 
 Thus, if $G_x^*$ is $3-$edge-colorable, then $G$ is $3-$edge-colorable; and a $3-$edge-coloring of $G$
can be easily extended to a $3-$edge-coloring of $G_x^*$. Therefore, we can conclude that {\bf (i)} is true.
 
 \item[(ii)] Clearly, $ (G\setminus \{xx^{'''},x^{'}x^{''}\}) \cup \{xx_0x_3x_4x_7x^{'''},x^{'}x_2x_1x_6x_5x^{''}\}$
  is  a spanning subgraph of $G_x^*$ homeomorphic to $G$.
  
  \item[(iii)] Let     
$M_x^*=M\cup E(H)\cup \{x_3x_5,x_4x_7\} $ 
 where $H$ is a copy of $K_{1,3}$ induced by $\{x_0,x_1,x_2,x_6\}$.
Obviously, $M_x^*$ is a planarizing perfect pseudo-matching of $G_x^*$ containing $M$ (see Figure~\ref{FIG:edge-crossing}). 

\item[(iv)]
Let $\mathcal{C}$ be a CDC of $G$ and let $C_1,C_2,C_3$, and $C_4$ be some cycles in $\mathcal{C}$ 
such that $xx^{'''}\in E(C_1)\cap E(C_2)$, and $x^{'}x^{''}\in E(C_3)\cap E(C_4)$.

 Put $P_1=xx_0x_3x_4x_7x^{'''}$,
$P_2=xx_0x_1x_6x_7x^{'''}$, $P_3=x^{'}x_2x_1x_6x_5x^{''}$, and $P_4=x^{'}x_2x_4x_3x_5x^{''}$. Now set
$C_1^{'}=C_1\setminus \{xx^{'''}\}\cup P_1$, $C_2^{'}=C_2\setminus \{xx^{'''}\}\cup P_2$, 
$C_3^{'}=C_3\setminus \{x^{'}x^{''}\}\cup P_3$,
$C_4^{'}=C_4\setminus \{x^{'}x^{''}\}\cup P_4$, and $C^{'}=x_0x_1x_2x_4x_7x_6x_5x_3x_0$.

Note that if some cycle in $\mathcal{C}$ covers both edges $xx^{'''}$ and $x^{'}x^{''}$, then suppose that $C_1=C_3$
and that likewise $C_2=C_4$ if $\mathcal{C}$ contains two cycles traversing $xx^{'''}$ and $x^{'}x^{''}$ each.
In this case, put $C_1^{'}=C_3^{'}=C_1\setminus \{xx^{'''},x^{'}x^{''}\}\cup P_1\cup P_3$ if $C_1=C_3$ and set $C_2^{'}=C_4^{'}=C_2\setminus \{xx^{'''},x^{'}x^{''}\}\cup P_2\cup P_4$ if $C_2=C_4$.

It follows that 
$$\mathcal{C}_x^*=( \mathcal{C}\setminus \{  C_1,C_2,C_3,C_4\} )  \cup   \{ C_1^{'},C_2^{'},C_3^{'},C_4^{'}, C^{'}\}  $$

 is a CDC of $G_x^*$.
\end{description}
\vspace{-.8cm}
\end{proof}

With the help of Theorem~\ref{Th:drawing} and Lemma~\ref{LEM:G*}, we can prove the following.
\begin{theorem}\label{THM:G*}
Given a cyclically $4-$edge-connected cubic graph $G$ with a  perfect pseudo-matching $M$. Then there exists a cubic graph $G^{*}$ with the following properties.

\begin{description}

\item[(i)] $G^*$ is a snark if and only if $G$ is a snark.

\item[(ii)] $G^*$ has a spanning subgraph homeomorphic to $G$.

\item[(iii)] $G^*$ has a planarizing perfect pseudo-matching $M^*$ with $M \subset M^*$; and moreover,
$G^*$ admits a  {\rm CDC}  containing the cycles of $G^*\setminus E(M^*)$.

\item[(iv)] Every {\rm CDC} of $G$ can be extended to a {\rm CDC} of $G^*$, but the converse is, unfortunately, not true. 

\end{description}
\end{theorem}
\begin{proof}
Suppose $M$ is a  perfect pseudo-matching of $G$.
 By Lemma~\ref{LEM:drawing}, there exists  a drawing with  $M$-avoiding intersections of $G$. 
Let $G^*$ be the graph, obtained from $G$ by repeatedly applying the same conversion  we did in the proof of Lemma~\ref{LEM:G*}, for all  edge-crossings in
$ E(G)\setminus E(M)$ in this drawing with  $M$-avoiding intersections of $G$.
Therefore, by repeatedly using  Lemma~\ref{LEM:G*}, we can check that all statements in Theorem~\ref{THM:G*} are true:
in particular, $G^*$ is cyclically $4-$edge-connected since $G$ is. Thus, (i) in Lemma~\ref{LEM:G*} translates into (i) in Theorem~\ref{THM:G*}. Furthermore, since $M^*$ is planarizing, $(G^*/M^*,\mathcal{T}_{M^*})$ has a compatible cycle decomposition which can be 
readily translated into a CDC of $G^*$ containing the cycles of $G^*\setminus E(M^*)$.
Finally, it is straightforward to see that $G^*$ may have a CDC which cannot be transformed into a CDC of $G$.
\end{proof}

One is tempted to improve Theorem~\ref{THM:G*} by using a perfect matching in $G$ and using the 
Blanu\v{s}a block  $B_2$ instead of $B_0$ for the crossings in a corresponding drawing of $G$. The larger graph $G^{*}$
would, in fact, contain  a planarizing perfect matching (see the perfect matching of $B_2$ in $B_2^2$ in Figure~\ref{FIG:BlanusaSnarks}).
However, due to the $3-$edge-coloring of $B_2$  we cannot draw the same conclusions as for $G^{*}$ in Theorem~\ref{THM:G*}.

Finally we observe that the construction of a cyclically $4-$edge-connected cubic graph with planarizing perfect pseudo-matching as expressed by Theorem~\ref{THM:G*}, tells us that there are at least as many snarks with planarizing perfect pseudo-matching as there are cyclically $5-$edge-connected snarks. `As many' is to be understood in terms of infinite  cardinalities of sets. This is expressed in our final theorem.   

\begin{theorem}\label{THM:map}
 Given the family $\mathcal{F}_5$ of all cyclically $5-$edge-connected snarks together with a drawing in the plane in accordance with Lemma~$\ref{LEM:drawing}$, for each element $G \in \mathcal{F}_5$. Call this drawing also $G$ and construct $G^*$ from $G$ in accordance with Theorem~$\ref{THM:G*}$. Define a mapping 

                                                                     $$ f:  \mathcal{F}_5  \longrightarrow  \mathcal{F}_4$$      
{\rm (}the latter denoting the family of cyclically $4-$edge-connected snarks having a planarizing 
perfect pseudo-matching{\rm )} by setting 

                                                                       $$f(G) = G^*.$$ 

It follows that $f$ is injective. 
\end{theorem}                                           

\begin{proof}  Let $G$ and $H$ be cyclically $5-$edge-connected snarks together with respective drawings in the plane in accordance with Lemma~\ref{LEM:drawing}, and let $G^*$ and $H^*$ be constructed from $G$ and $H$, respectively, in accordance with Theorem~\ref{THM:G*}. Suppose $G^*$ and $H^*$ are isomorphic; let $h^*(G^*) = H^*$ be such an isomorphism. It follows that the only cyclic $4-$edge-cuts in $G^*, H^*$, respectively, are of the form %
                                          $$F=\{x_0x,x_5x^{''},x_2x^{'},x_7x^{'''}\}$$
(see Figure~\ref{FIG:edge-crossing}), separating a Blanusa block  $B_0$  from the rest of $G^*, H^*$, respectively. We    denote these edge-cuts in $G^*, H^*$, respectively, by
                    $$F_{G^*} = \{x_{0_{G^*}}x_{ _{G^*}}, x_{5_{G^*}}x^{''}_{ _{G^*}}, x_{2_{G^*}}x^{'}_{ _{G^*}}, x_{7_{G^*}}x^{'''}_{ _{G^*}}\}$$ 
and 
 $$F_{H^*} = \{x_{0_{H^*}}x_{ _{H^*}}, x_{5_{H^*}}x^{''}_{ _{H^*}}, x_{2_{H^*}}x^{'}_{ _{H^*}}, x_{7_{H^*}}x^{'''}_{ _{H^*}}\}.$$

Now, assume without loss of generality that $h^*(F_{G^*}) = F_{H^*}$. It follows that                                   
     $$h^*(\{x_{0_{G^*}}x_{ _{G^*}}, x_{7_{G^*}}x_{ _{G^*}}^{'''}\})  =  \{x_{0_{H^*}}x_{ _{H^*}}, x_{7_{H^*}}x^{'''}_{ _{H^*}}\}$$ 
and  
$$h^*(\{x_{5_{G^*}}x^{''}_{ _{G^*}}, x_{2_{G^*}}x^{'}_{ _{G^*}}\})  =  \{x_{5_{H^*}}x^{''}_{ _{H^*}}, x_{2_{H^*}}x^{'}_{ _{H^*}}\}$$
(otherwise we could rotate the notation). This leads to an isomorphism $h$ between $G$ and $H$
by setting                  $h(e) = h^*(e)$  for every edge $e$ which is not an element of a pair of crossing edges;

and set
                $h(x_Gx^{'''}_G) = x_Hx^{'''}_H$ and $h(x^{'} _{G}x^{''}_G )= x^{'}_Hx^{''}_H $
for every pair of crossing edges in the drawing of $G$ and $H$, respectively. In other words, non-isomorphic cyclically $5-$edge-connected snarks correspond to non-isomorphic cyclically $4-$edge-connected snarks with planarizing perfect pseudo-matching. That is, $f$ is injective.
\end{proof}

\section{Final remarks}

We have demonstrated the usefulness of the concept of perfect pseudo-matchings $M$ in categorizing snarks $G$, be it in connection with compatible cycle decompositions in $G/M$, be it – more generally – in connection with cycle double covers. In particular, we showed that when aiming at solving the CDC Conjecture we may restrict ourselves to snarks with stable dominating cycle (Question~\ref{Q:1} and Proposition~\ref{Prop:1}); or alternatively, we only need to deal with snarks having no SUD-$K_5-$minor-free perfect pseudo-matching. 

As for several well-known classes of snarks we have shown that they have planarizing perfect pseudo-matchings, and that planarizing perfect pseudo-matchings appear in many snarks (Theorems~\ref{THM:G*} and~\ref{THM:map}). As a matter of fact, an old conjecture of the first author of this paper claims implicitly that if a snark $G$  has a planarizing perfect pseudo-matching which corresponds to a dominating cycle $C$ in $G$, then $G$ has a $5-$cycle double  cover containing $C$ (see~\cite[Conjecture 10]{Fleischner1988}). This would lend support to Hoffmann-Ostenhof's Strong $5-$Cycle Double Cover Conjecture (see~\cite{Arthur}).

\end{document}